\documentclass{amsart}
\usepackage{amssymb}
\usepackage{amsmath, amscd}
\usepackage{amsthm}
\usepackage{mathrsfs}
\usepackage[all,cmtip]{xy}
\usepackage[foot]{amsaddr}
\usepackage[T1]{fontenc}
\usepackage{verbatim}
\usepackage{mathrsfs}
\usepackage{enumerate}
\usepackage{mathtools}

\numberwithin{equation}{subsection}

\newtheorem{theorem}{Theorem}[subsection]
\newtheorem{lemma}[theorem]{Lemma}

\newtheorem{corollary}[theorem]{Corollary}
\newtheorem{definition}[theorem]{Definition}

\newtheorem{proposition}[theorem]{Proposition}

\newtheorem*{thm1}{Theorem 1}

\newtheorem*{cor1}{Corollary 1}

\theoremstyle{remark}

\newtheorem{rmk}[theorem]{Remark}

\newcommand{\GZip}{\mathop{\text{$G$-{\tt Zip}}}\nolimits}

\newcommand{\B}{\mathop{\text{{\tt Brh}}}\nolimits}

\newcommand{\FZip}{\mathop{\text{$F$-{\tt Zip}}}\nolimits}

\newcommand{\GF}{\mathop{\text{$G$-{\tt ZipFlag}}}\nolimits}

\newskip\procskipamount
\newskip\interskipamount
\newskip\refskipamount
\newcommand{\procskip}{\vskip\procskipamount}
\newcommand{\interskip}{\vskip\interskipamount}
\newcommand{\refskip}{\vskip\refskipamount}

\newcommand{\procbreak}{\par
   \ifdim\lastskip<\procskipamount\removelastskip
   \penalty-100
   \procskip\fi
   \noindent\ignorespaces}

\newcommand{\titlebreak}{\par%
\ifdim\lastskip<\interskipamount\removelastskip%
\penalty10000%
\interskip\fi%
\noindent}%

\newcommand{\interbreak}{\par%
\ifdim\lastskip<\interskipamount\removelastskip%
\penalty-100%
\interskip\fi%
\noindent\ignorespaces}%

\newcommand{\refbreak}{\par%
\ifdim\lastskip<\refskipamount\removelastskip%
\penalty-100%
\refskip\fi%
\noindent\ignorespaces}%

\newcounter{listcounter}
\newcounter{deflistcounter}
\newcounter{equivcounter}

\newskip{\itemsepamount}
\newskip{\topsepamount}
\newenvironment{assertionlist}{%
  \begin{list}
    {\upshape (\arabic{listcounter})}
    {\setlength{\leftmargin}{18pt}
     \setlength{\rightmargin}{0pt}
     \setlength{\itemindent}{0pt}
     \setlength{\labelsep}{5pt}
     \setlength{\labelwidth}{13pt}
     \setlength{\listparindent}{\parindent}
     \setlength{\parsep}{0pt}
     \setlength{\itemsep}{\itemsepamount}
     \setlength{\topsep}{\topsepamount}
     \usecounter{listcounter}}}
  {\end{list}}

\newenvironment{definitionlist}{%
  \begin{list}
    {\upshape (\alph{deflistcounter})}
    {\setlength{\leftmargin}{18pt}
     \setlength{\rightmargin}{0pt}
     \setlength{\itemindent}{0pt}
     \setlength{\labelsep}{5pt}
     \setlength{\labelwidth}{13pt}
     \setlength{\listparindent}{\parindent}
     \setlength{\parsep}{0pt}
     \setlength{\itemsep}{\itemsepamount}
     \setlength{\topsep}{\topsepamount}
     \usecounter{deflistcounter}}}
  {\end{list}}

\newenvironment{equivlist}{%
  \begin{list}
    {\upshape (\roman{equivcounter})}
    {\setlength{\leftmargin}{18pt}
     \setlength{\rightmargin}{0pt}
     \setlength{\itemindent}{0pt}
     \setlength{\labelsep}{5pt}
     \setlength{\labelwidth}{13pt}
     \setlength{\listparindent}{\parindent}
     \setlength{\parsep}{0pt}
     \setlength{\itemsep}{\itemsepamount}
     \setlength{\topsep}{\topsepamount}
     \usecounter{equivcounter}}}
  {\end{list}}

% ...
\newcommand{\Fcal}{{\mathcal F}}

\newcommand{\Ocal}{{\mathcal O}}

\newcommand{\Vcal}{{\mathcal V}}

\newcommand{\Zcal}{{\mathcal Z}}

\newcommand{\bfr}{{\mathfrak b}}

\renewcommand{\AA}{\mathbf{A}}

\newcommand{\DD}{\mathbf{D}}

\newcommand{\FF}{\mathbf{F}}
\newcommand{\GG}{\mathbf{G}}

\newcommand{\NN}{\mathbf{N}}

\newcommand{\QQ}{\mathbf{Q}}

\newcommand{\XX}{\mathbf{X}}

\newcommand{\ZZ}{\mathbf{Z}}

\newcommand{\Xscr}{{\mathscr X}}
\newcommand{\Yscr}{{\mathscr Y}}

\newcommand{\rk}{\mbox{ rk }}

\newcommand{\Th}{{\rm Th.}}

\newcommand{\Cor}{{\rm Cor.}}

\newcommand{\Lem}{{\rm Lem.}}

\newcommand{\Def}{{\rm Def.}}

\newcommand{\Prop}{{\rm Prop.}}

\newcommand{\loccit}{{\em loc.\ cit. }}
\newcommand{\loccitn}{{\em loc.\ cit.}}

\DeclareMathOperator{\Fl}{Fl}

\renewcommand{\Im}{{\rm Im}}

\DeclareMathOperator{\Lie}{Lie}

\DeclareMathOperator{\Adj}{Ad}

\DeclareMathOperator{\Ker}{Ker}

\DeclareMathOperator{\Span}{Span}
\DeclareMathOperator{\Stab}{Stab}
\DeclareMathOperator{\Spec}{Spec}

\begin{document}

\title{Normalization of closed Ekedahl-Oort strata}
\author{Jean-Stefan Koskivirta}
\date{\today}
\address{J.-S. K. : Department of Mathematics, Imperial College London}
\email{jeanstefan.koskivirta@gmail.com}
\pagestyle{plain}

\begin{abstract}
We apply our theory of partial flag spaces developed in \cite{Goldring-Koskivirta-zip-flags} to study a group-theoretical generalization of the canonical filtration of a truncated Barsotti-Tate group of level 1. As an application, we determine explicitly the normalization of the Zariski closures of Ekedahl-Oort strata of Shimura varieties of Hodge-type as certain closed coarse strata in the associated partial flag spaces.
\end{abstract}
\maketitle
\section*{Introduction}

Let $H$ be a truncated Barsotti-Tate group of level 1 (in short BT1),  over an algebraically closed field $k$ of characteristic $p$, and let $\sigma:k\to k$ denote the map $x\mapsto x^p$. Denote by $D:=\DD(H)$ its Dieudonn\'{e} module, which is a finite-dimensional $k$-vector space $D$ endowed with a $\sigma$-linear endomorphism $F$ and a $\sigma^{-1}$-linear endomorphism $V$ satisfying $FV=VF=0$ and $\Im(F)=\Ker(V)$. Oort showed in \cite{Oort-stratification-moduli-space-abelian-varieties} that there exists a flag of $D$ which is stable by $V$ and $F^{-1}$ and which is coarsest among all such flags, called the canonical filtration of $D$. After choosing a basis of $D$, we obtain a filtration of $k^n$ (where $n=\dim_k(D)$ is the height of $H$). The stabilizer of this flag is a parabolic subgroup $P(H)\subset GL_n$, well-defined up to conjugation. We want to emphasize that this construction attaches a group-theoretical object $P(H)$ to a truncated Barsotti-Tate group of level 1.

The theory of $F$-zips developed in \cite{Moonen-Wedhorn-Discrete-Invariants}, \cite{Pink-Wedhorn-Ziegler-zip-data} and \cite{PinkWedhornZiegler-F-Zips-additional-structure} establishes the precise link between BT1's and group theory. Specifically, isomorphism classes of BT1's of height $n$ and dimension $d$ correspond bijectively to $E$-orbits in $GL_n$, where $E$ is the zip group (see section \ref{subsec Dieudonne zips}). The stack of $F$-zips of type $(n,d)$ can be defined as the quotient stack $\FZip^{n,d}=\left[E\backslash GL_n\right]$. Moreover, there is a natural morphism of stacks $BT^{n,d}_1 \to \FZip^{n,d}$, where $BT^{n,d}_1$ is the stack of BT1's of height $n$ and dimension $d$ over $k$. More generally, let $G$ be a connected reductive group over $\FF_p$, and $P,Q\subset G$ parabolic subgroups (defined over some finite extension of $\FF_p$). Let $L\subset P$ and $M\subset Q$ be Levi subgroups and assume that $\varphi(L)=M$, where $\varphi:G\to G$ is the Frobenius homomorphism. One can define the stack of $G$-zips of type $\Zcal=(G,P,L,Q,M,\varphi)$ as the quotient stack $\GZip^\Zcal=\left[E\backslash G\right]$, where $E\subset P\times Q$ is the zip group (see section \ref{subsec Gzip}). For example, if $G$ is the automorphism group of a PEL-datum, then $\GZip^\Zcal(k)$ classifies BT1's over $k$ of type $\Zcal$ endowed with this additional structure. If $W$ denotes the Weyl group of $G$, the $E$-orbits in $G$ are parametrized by a subset ${}^I W\subset W$ (see section \ref{subsec stratif}). Denote by $G_w\subset G$ the $E$-orbit corresponding to $w\in {}^I W$ and put $Z_w:=[E\backslash G_w]$ the corresponding zip stratum.

In \cite{Goldring-Koskivirta-zip-flags}, we defined for each parabolic $P_0\subset P$ the stack of partial zip flags $\GF^{(\Zcal,P_0)}$ endowed with a natural projection $\pi:\GF^{(\Zcal,P_0)}\to \GZip^\Zcal$ which makes it a $P/P_0$-bundle over $\GZip^\Zcal$. This defines a tower of stacks above $\GZip^\Zcal$. Moreover, the stack $\GF^{(\Zcal,P_0)}$ admits two natural stratifications. In general, one is finer than the other, but they coincide when $P_0$ is a Borel subgroup. The fine strata $Z_{P_0,w}$ are parametrized by $w\in {}^{I_0}W$, where ${}^{I_0}W \subset W$ is a subset containing ${}^I W$ (see section \ref{subsec fine coarse}). The strata $Z_{P_0,w}$ attached to elements $w\in {}^IW$ are called minimal and satisfy $\pi(Z_{P_0,w})=Z_w$ and the restriction $\pi:Z_{P_0,w}\to Z_w$ is finite \'{e}tale. When $P_0=P$, the stack $\GF^{(\Zcal,P)}$ coincides with $\GZip^\Zcal$ and the fine stratification is the stratification by $E$-orbits, whereas the coarse stratification is given by $P\times Q$-orbits. In general, we say that a stratum $Z_{P_0,w}$ has coarse closure if it is open in the coarse stratum containing it. If $Z_{P_0,w}$ has coarse closure, its Zariski closure  $\overline{Z}_{P_0,w}$ is normal.

In the formalism of $G$-zips, one can attach to each $w\in {}^I W$ a parabolic subgroup $P_w\subset P$. In the case $G=GL_n$, if $H$ is a BT1 corresponding to $w\in {}^IW$ under the correspondence between BT1's and $E$-orbits, then $P_w$ is the parabolic $P(H)$ defined above. This is proved in \Prop~\ref{can prop}. Since $P_w$ is canonically attached to $w$, it is natural to ask what special property is satisfied by the stratum $Z_{P_w,w}$ of $\GF^{(\Zcal,P_w)}$. Our main theorem answers this question:

\begin{thm1}\label{thm intro}[Th.~\ref{mainthm}]
Let $w\in {}^I W$. The following properties hold:
\begin{enumerate}[(i)]
\item \label{item isom intro}  $\pi:Z_{P_w,w}\to Z_w$ is an isomorphism.
\item \label{item coarse intro}  $Z_{P_w,w}$ has coarse closure.
\end{enumerate}
Furthermore, among all parabolic subgroups $P_0$ such that ${}^zB\subset P_0\subset P$, the parabolic $P_w$ is the smallest parabolic satisfying \eqref{item isom intro} and the largest one satisfying \eqref{item coarse intro}.
\end{thm1}

The Borel subgroup ${}^z B$ of the above theorem is defined in section \ref{subsec frame}. Note that property \eqref{item isom intro} is obviously satisfied for $P_0=P$ and property \eqref{item coarse intro} is satisfied for $P_0={}^zB$ because fine and coarse strata coincide in this case. Hence the canonical parabolic $P_w$ is the unique intermediate parabolic such that both properties are satisfied. As a consequence, we deduce that the normalization of the Zariski closure $\overline{Z}_w$ is $\tilde{Z}_w:=\Spec(\Ocal(\overline{Z}_{P_w,w}))$ (see Corollary \ref{cor norma GZip}).

Let $X$ be the special fiber of a good reduction Hodge-type Shimura variety, and let $G$ be the attached reductive $\FF_p$-group (see section \ref{subsec Shimura}). In \cite{ZhangEOHodge}, Zhang has constructed a smooth map of stacks
\begin{equation*}
\zeta:X\longrightarrow \GZip^\Zcal
\end{equation*}
where $\Zcal$ is the zip datum attached to $X$ as in \cite{Goldring-Koskivirta-zip-flags}~\S6.2. The Ekedahl-Oort stratification of $X$ is defined as the fibers of $\zeta$. For $w\in {}^IW$, set $X_{w}:=\zeta^{-1}(Z_{w})$.
For any parabolic ${}^z B \subset P_0 \subset P$, define the partial flag space $X_{P_0}$ as the fiber product
\begin{equation*}
\xymatrix@1@M=5pt@C=30pt@R=15pt{
X_{P_0} \ar[r]^-{\zeta_{P_0}} \ar[d]_-{\pi} & \GF^{(\Zcal, P_0)} \ar[d]^-{\pi_{P_0}} \\
X \ar[r]_-{\zeta} & \GZip^{\Zcal}}
\end{equation*} 
For $w\in {}^{I_0}W$, define the fine stratum $X_{P_0,w}:=\zeta_{P_0}^{-1}(Z_{P_0,w})$ of $X_{P_0}$. The space $X_{P_0}$ is a generalization of the flag space considered by Ekedahl and Van der Geer in \cite{EkedahlGeerEO}, where they consider flags refining the Hodge filtration of an abelian variety.

\begin{cor1}\label{cor1 intro}[Cor.~\ref{cor norma Shimura}]
Let $w\in {}^IW$. The normalization of $\overline{X}_{w}$ is the Stein factorization of the map $\pi: \overline{X}_{P_w,w}\to\overline{X}_{w}$. It is isomorphic to $\overline{X}_{w}\times_{\GZip^\Zcal}\tilde{Z}_w$.
\end{cor1}

For Siegel-type Shimura varieties, an analogous result to \Cor~1 was proved by Boxer in \cite[\Th 5.3.1]{Boxer-thesis} using different methods.

We now give an overview of the paper. In section 1, we review the theory of $G$-zips and prove a result on point stabilizers for later use. Section 2 is devoted to the stack of partial $G$-zips and its stratifications. We define minimal strata and give an explicit form for the restriction of the map $\pi$ to a minimal stratum (\Prop~\ref{comdiag}). In section 3, we define the notion of canonical parabolic and explain its relevance with respect to the normalization of a closed stratum of $\GZip^\Zcal$. We prove Theorem \ref{mainthm} after giving criteria for properties \eqref{item isom intro} and \eqref{item coarse intro} above. Finally, we explain in section 4 the correspondence between the classical theory of BT1's and the theory of G-zips, following \cite{Pink-Wedhorn-Ziegler-zip-data}. We establish the link between the parabolic $P_w$ and the canonical parabolic of a BT1.

\subsection*{Acknowledgements}
The author would like to thank Wushi Goldring for many fruitful discussions about this work and Torsten Wedhorn for helpful comments on this paper. I am grateful to the reviewer for suggesting improvements.

\section{Review of $G$-zips}\label{review}

We will need to review some facts about the stack of $G$-zips found in \cite{Pink-Wedhorn-Ziegler-zip-data} and prove a result on the stabilizer of an element by the group $E$.

\subsection{The stack $\GZip^\Zcal$} \label{subsec Gzip}
We fix an algebraic closure $k$ of $\FF_p$. A zip datum is a tuple $\Zcal=(G,P,L,Q,M,\varphi)$ where $G$ is a connected reductive group over $\FF_p$, $\varphi:G\to G$ is the Frobenius homomorphism, $P,Q\subset G$ are parabolic subgroups of $G_k$, $L\subset P$ and $M\subset Q$ are Levi subgroups of $P$ and $Q$ respectively. One imposes the condition $\varphi(L)=M$. One can attach to $\Zcal$ a zip group $E$ defined by
\begin{equation}\label{zipgroup}
E:=\{(p,q)\in P\times Q, \varphi(\overline{p})=\overline{q}\}
\end{equation}
where $\overline{p}\in L$ and $\overline{q}\in M$ denote the projections of $p$ and $q$ via the isomorphisms $P/R_u(P)\simeq L$ and $Q/R_u(Q)\simeq M$. We let $G\times G$ act on $G$ via $(a,b)\cdot g:=agb^{-1}$ and we obtain by restriction an action of $E$ on $G$. The stack of $G$-zips is then isomorphic to the following quotient stack:
\begin{equation}
\GZip^\Zcal\simeq \left[E \backslash G \right].
\end{equation}
When we want to specify the zip datum $\Zcal$, we write sometimes $E_\Zcal$ for the zip group $E$.

\subsection{Frame}\label{subsec frame}
A frame for $\Zcal$ is a triple $(B,T,z)$ where $(B,T)$ is a Borel pair and $z\in G(k)$ satisfying the following conditions:
\begin{enumerate}[(i)]
\item $B\subset Q$
\item ${}^z T\subset L$
\item ${}^z B \subset P$
\item $\varphi({}^z B\cap L)=B\cap M$
\item $\varphi({}^z T)= T$.
\end{enumerate}
We fix throughout a frame $(B,T,z)$ and we define:
\begin{enumerate}[(1)]
\item $\Phi\subset X^*(T)$ : the set of $T$-roots of $G$.
\item $\Phi_+$ : the set of positive roots with respect to $B$.
\item $\Delta\subset \Phi_+$ : the set of positive simple roots.
\item For $\alpha \in \Phi$, let $s_\alpha \in W$ be the corresponding reflection. Then $(W,\{s_\alpha\}_{\alpha \in \Delta})$ is a Coxeter group and we denote by $\ell :W\to \NN$ the length function.
\item For $K\subset \Delta$, Let $W_K \subset W$ be the subgroup generated by the $s_\alpha$ for $\alpha \in K$. Let $w_0\in W$ be the longest element and $w_{0,K}$ the longest element in $W_K$.
\item If $R\subset G$ is a parabolic subgroup containing $B$ and $D$ is the unique Levi subgroup of $R$ containing $T$, then the type of $R$ (or of $D$) is the unique subset $K\subset \Delta$ such that $W(D,T)=W_K$. The type of an arbitrary parabolic $R$ is the type of its unique conjugate containing $B$. Let $I\subset \Delta$ (resp. $J\subset \Delta$) be the type of $P$ (resp. $Q$).

\item For $K\subset \Delta$, ${}^K W$ (resp. $W^K$) : the subset of elements $w\in W$ which are minimal in the coset $W_K w$ (resp. $wW_K$).
\item For $K,R\subset \Delta$, ${}^K W^R:={}^K W \cap W^R$.
\item For an element $x\in {}^I W^J$, define $I_x:=J\cap {}^{x^{-1}}I$. By Proposition 2.7 in \cite{Pink-Wedhorn-Ziegler-zip-data}, any element $w\in W_IxW_J$ can be uniquely written as
\begin{equation}\label{uniqueform}
w=xw_J, \quad \textrm{with} \quad w_J\in {}^{I_x}W_J.
\end{equation}
\end{enumerate}

For $w\in W$, one has an equivalence:
\begin{equation}\label{IWeq}
w\in {}^IW \Longleftrightarrow  {}^zB\cap L = {}^{zw}B \cap L.
\end{equation}

\subsection{Stratification} \label{subsec stratif}
For $w\in W$, choose a representative $\dot{w}\in N_G(T)$, such that $(w_1w_2)^\cdot = \dot{w}_1\dot{w}_2$ whenever $\ell(w_1 w_2)=\ell(w_1)+\ell(w_2)$ (this is possible by choosing a Chevalley system, see \cite{SGA3}, Exp. XXIII, \S6). For $h\in G(k)$, denote by $\Ocal_\Zcal(h)$ the $E$-orbit of $h$ in $G$ and define $\mathfrak{o}_\Zcal(h):=[E\backslash \Ocal_\Zcal(h)]$. By Theorem 7.5 in \cite{Pink-Wedhorn-Ziegler-zip-data}, there is a bijection:
\begin{equation}\label{param}
{}^I W \to \{E \textrm{-orbits in }G\}, \quad w\mapsto G_w:=\Ocal(z\dot{w}).
\end{equation}
Furthermore, for all $w\in {}^I W $, one has:
\begin{equation}\label{dimzipstrata}
\dim(G_w)= \ell(w)+\dim(P).
\end{equation}
For $w\in {}^I W$, we endow the locally closed subset $G_w$ with the reduced structure, and we define the corresponding zip stratum of $\GZip^\Zcal$ by $Z_w:= \left[E\backslash G_w \right]$.

\subsection{Point stabilizers}\label{subsec stab}
\begin{definition}[{\cite[\Def~5.1]{Pink-Wedhorn-Ziegler-zip-data}}]\label{Mw def}
Let $w\in {}^I W$. There is a largest subgroup $M_w$ of ${}^{w^{-1}z^{-1}}L$ satisfying $\varphi({}^{zw}M_w)=M_w$.
\end{definition}
In \loccit \S 5.1, this subgroup is denoted by $H_w$. It is a Levi subgroup of $G$ contained in $M$. We define also:
\begin{align}
L_w&:={}^{zw}M_w \subset L  \label{canlevi} \\
P_w&:=L_w{}^z B \subset P \\ \label{canparab}
Q_w&:=M_w B \subset Q
\end{align}
Since $\varphi(L_w)=M_w$, we obtain a zip datum $\Zcal_w:=(G,P_w,L_w,Q_w,M_w,\varphi)$. Note that $(B,T,z)$ is again a frame for $\Zcal_w$. If an algebraic group $G$ acts on a $k$-scheme $X$ and $x\in X(k)$, we denote by $\Stab_G(x)$ the scheme-theoretical stabilizer of $x$. For an algebraic group $H$, we denote by $H_{\rm red}$ the underlying reduced algebraic group and by $H^\circ$ the identity component of $H$.

\begin{lemma} \label{ptstab} \
\begin{enumerate}
\item \label{item stab red} One has $\Stab_{E}(z\dot{w})_{\rm red}= A \ltimes R$ where $A\subset L_w \times M_w$ is the finite group
\begin{equation}\label{groupA}
A:=\{(x,\varphi(x)), \ x\in L_w, \ {}^{z\dot{w}}\varphi(x)=x\}
\end{equation}
and $R$ is a unipotent smooth connected normal subgroup.
\item \label{item stab idcomp} One has $\Stab_E(z\dot{w})^\circ \subset {}^zB\times B$.
\end{enumerate}
\end{lemma}

\begin{proof}
The first part is \Th~8.1 in \loccit To prove \eqref{item stab idcomp}, it suffices to show that $\Stab_E(z\dot{w})^\circ \subset {}^zB\times G$, or equivalently $\Lie(\Stab_E(z\dot{w}))\subset \Lie({}^zB)\times \Lie(G)$. We follow the proof of \Th~8.5 of \loccit An arbitrary tangent vector of $E$ at $1$ has the form $(1+dp,1+dv)$ for $dp\in \Lie(P)$ and $dv\in \Lie(V)$. This element stabilizes $z\dot{w}$ if and only if $dp=\Adj_{z\dot{w}}(dv)$. Hence 
\begin{equation}
dp\in \Lie(P)\cap \Adj_{z\dot{w}}(\Lie(V))=\Lie(P\cap {}^{zw}V)
\end{equation}
Hence it suffices to show $P\cap {}^{zw}V\subset {}^zB$. This amounts to $L\cap {}^{zw}V\subset L\cap {}^zB$ and equivalently $M\cap \varphi({}^{zw}V)\subset M\cap\varphi({}^zB)=M\cap B$. This is proved in \Prop 4.12 of \loccit More precisely, the authors define in construction 4.3 a group $V_x$ (note that the element $z$ is denoted by $g$ there), where $w=xw_J$ is a decomposition as in \eqref{uniqueform}. One has $V_{x}=M\cap \varphi({}^{z\dot{x}}V)=M\cap \varphi({}^{zw}V)$ because $w_J\in W_J$, so ${}^{w_J}V=V$. \Prop~4.12 of \loccit shows that $(M\cap B,T,1)$ is a frame for $\Zcal_{\dot{x}}$, so in particular one has $V_x\subset M\cap B$. This terminates the proof of the lemma.

\end{proof}

\section{The stack of partial zip flags} \label{sec zip flags}
We recall in this section some of the results of \cite{Goldring-Koskivirta-zip-flags}.
\subsection{Fine and coarse flag strata}\label{subsec fine coarse}
For each parabolic subgroup $P_0$ satisfying ${}^z B \subset P_0\subset P$, we defined in \loccit\S2 a stack $\GF^{(\Zcal,P_0)}$ which parametrizes $G$-zips of type $\Zcal$ endowed with a compatible $P_0$-torsor. There is an isomorphism
\begin{equation}
\GF^{(\Zcal,P_0)}\simeq \left[E \backslash \left(G\times P/P_0 \right) \right]
\end{equation}
where $E$ acts on $G\times P/P_0$ by $(a,b)\cdot (g,xP_0):=(agb^{-1},axP_0)$. Furthermore, there is a natural projection map $\pi:\GF^{(\Zcal,P_0)}\to \GZip^\Zcal$ which is a $P/P_0$-bundle.

Denote by $L_0\subset P_0$ the Levi subgroup containing ${}^z T$ (note that $L_0\subset L$). We define a second zip datum $\Zcal_0=(G,P_0,L_0,Q_0,M_0,\varphi)$ by setting: 
\begin{align}
M_0&:=\varphi(L_0)\subset M \label{mzero} \\
Q_0&:=M_0 B\subset Q.\label{qzero}
\end{align}
Note that $(B,T,z)$ is again a frame of $\Zcal_0$. By \loccit \S 3.1, there is a natural morphism of stacks
\begin{equation}
\Psi_{P_0}:\GF^{(\Zcal,P_0)}\to \GZip^{\Zcal_0}
\end{equation}
which is an $\AA^r$-bundle for $r=\dim(P/P_0)$ (Proposition 3.1.1 in \loccitn). It is induced by the map $G\times P \to G$, $(g,a)\mapsto a^{-1}g \varphi(\overline{a})$. Let $I_0$ and $J_0$ denote respectively the types of $P_0$ and $Q_0$. For $w\in {}^{I_0} W$, we define the fine flag stratum $Z_{P_0,w}$ of $\GF^{(\Zcal,P_0)}$ as the locally closed substack
\begin{equation}
Z_{P_0,w}:=\Psi_{P_0}^{-1}(\mathfrak{o}_{\Zcal_0}(z\dot{w}))
\end{equation}
endowed with the reduced structure. Explicitly, one has $Z_{P_0,w}=\left[E\backslash G_{P_0,w} \right]$ where $G_{P_0,w}$ is the algebraic subvariety of $ G\times P/P_{0}$ defined by
\begin{equation}
G_{P_0,w}:=\{(g,aP_0)\in G\times P/P_{0}, \ a^{-1}g\varphi(\overline{a})\in \Ocal_{\Zcal_0}(z\dot{w})\}.
\end{equation}
Denote by $\B^{\Zcal_0}$ the quotient stack $\left[P_0 \backslash G / Q_0 \right]$, called the Bruhat stack. Since $E_{\Zcal_0}\subset P_0\times Q_0$, there is a natural projection morphism $\beta:\GZip^{\Zcal_0}\to \B^{\Zcal_0}$. The composition $\Psi_{P_0}\circ \beta$ gives a smooth map of stacks
\begin{equation}
\psi_{P_0}:\GF^{(\Zcal,P_0)} \to \B^{\Zcal_0}.
\end{equation}
By \cite[\Lem~1.4]{Wedhorn-bruhat}, the set $\{z\dot{x}, x\in {}^{I_0}W^{J_0}\}$ is a set of representatives of the $P_0\times Q_0$-orbits in $G$ (pay attention to the fact that ${}^zB\subset P$). For $x\in {}^{I_0}W^{J_0}$, write $\bfr(x):=[P_0\backslash (P_0z\dot{x}Q_0)/Q_0]$ (locally closed substack of $\B^{\Zcal_0}$) and define the coarse flag stratum $\ZZ_{P_0,x}$ as
\begin{equation}
\ZZ_{P_0,x}:=\psi_{P_0}^{-1}(\bfr(x))
\end{equation}
endowed with the reduced structure. Explicitly, one has $\ZZ_{P_0,x}=\left[E\backslash \GG_{P_0,x} \right]$ where $\GG_{P_0,x}$ is the subvariety of $G\times P/P_0$ defined by
\begin{equation}
\GG_{P_0,x}:=\{(g,aP_0)\in G\times P/P_{0}, \ a g\varphi(\overline{a})^{-1}\in P_0z\dot{x}Q_0 \}.
\end{equation}

All fine and coarse flag strata are smooth. A coarse stratum is a union of fine strata and the Zariski closure of a coarse flag stratum is normal. In each coarse stratum there is a unique open fine stratum.

\begin{definition}
We say that a fine flag stratum $Z$ has coarse closure if it is open in the coarse stratum that contains it, equivalently, if its Zariski closure coincides with the Zariski closure of a coarse flag stratum.
\end{definition}
In particular, the Zariski closure $\overline{Z}$ of such a stratum is normal (\cite[\Prop~2.2.1(1)]{Goldring-Koskivirta-zip-flags}).

\subsection{Minimal strata}
Recall that we defined in \cite{Goldring-Koskivirta-zip-flags} a minimal flag stratum as a flag stratum $Z_{P_0,w}$ parametrized by an element $w\in {}^I W$. For a minimal stratum one has $\pi(Z_{P_0,w})=Z_w$ and the induced morphism $\pi:Z_{P_0,w}\to Z_w$ is finite (\Prop~3.2.2 of \loccitn). The following proposition shows that it is also \'{e}tale. For $w\in {}^IW$, denote by $\tilde{\pi}:G_{P_0,w}\to G_w$ the first projection, it is an $E$-equivariant map.

\begin{proposition}\label{comdiag}
Let $w\in {}^I W$ and denote by $S:=\Stab_{E}(z\dot{w})$ the stabilizer of $z\dot{w}$ in $E$ and define $S_{P_0}:=S \cap  (P_0\times G)$.
\begin{enumerate}[(1)]
\item \label{item prop comdiag} There is a commutative diagram
$$\xymatrix@1@M=3pt{
Z_{P_0,w} \ar[r]^-{\simeq} \ar[d]_{\pi} & \left[1/S_{P_0}  \right] \ar[d] \\
Z_{w} \ar[r]^-{\simeq}  & \left[1/S \right]}$$
where the horizontal maps are isomorphisms and the right-hand side vertical map is the natural projection.
\item \label{item etale} The map $\pi:Z_{P_0,w}\to Z_w$ is finite \'{e}tale.
\item \label{item pi isom} The map $\pi:Z_{P_0,w}\to Z_w$ is an isomorphism if and only if the inclusion $S\subset P_0 \times G$ holds.
\end{enumerate}
\end{proposition}

\begin{proof}
We first prove \eqref{item prop comdiag}. There is a natural identification $G_w\simeq \left[E/S \right]$ because $G_w$ is the $E$-orbit of $z\dot{w}$. It follows that $Z_w\simeq [E\backslash E/S]\simeq [1/S]$. Similarly, we claim that the variety $G_{P_0,w}$ consists of a single $E$-orbit. This was proved in \cite{Goldring-Koskivirta-Strata-Hasse-v2} \Prop~5.4.5 in the case when $P_0$ is a Borel subgroup. For a general $P_0$, we may reduce to the Borel case as follows: By Proposition 3.2.2 of \cite{Goldring-Koskivirta-zip-flags}, we have a natural $E$-equivariant surjective projection map $G_{{}^zB,w}\to G_{P_0,w}$, hence $G_{P_0,w}$ consists of a single $E$-orbit. We thus can identify $Z_{P_0,w}\simeq \left[E \backslash E/S' \right]$ where $S'=\Stab_E(z\dot{w},1)$. It is clear that $S'=S_{P_0}$, so the result follows.

We now show \eqref{item etale}. By \Prop~3.2.2 of \loccitn, we know that $\pi:Z_{P_0,w}\to Z_w$ is finite. By \eqref{item prop comdiag}, it is equivalent to show that $S/S_{P_0}$ is an \'{e}tale scheme. By Lemma \ref{ptstab}~\eqref{item stab idcomp}, we have $S^\circ \subset S\cap({}^zB\times G) \subset S_{P_0}$. Hence the quotient map $S\to S/S_{P_0}$ factors through a surjective map $\pi_0(S)\to S/S_{P_0}$, which shows that $S/S_{P_0}$ is \'{e}tale.

Finally, the last assertion follows immediately from \eqref{item prop comdiag}.
\end{proof}

\section{The canonical parabolic}
We fix an element $w\in {}^I W$. Recall that we defined in \eqref{canparab} a parabolic subgroup $P_w \subset P$.
\subsection{Definition}
Let $P_0$ be a parabolic subgroup of $G$ such that ${}^z B \subset P_0 \subset P$.
\begin{definition}\label{canonical}
We say that $P_0$ is a canonical parabolic subgroup for $w$ if the following properties are satisfied:
\begin{enumerate}[(i)]
\item \label{item isom} The map $\pi:Z_{P_0,w}\to Z_w$ is an isomorphism.
\item \label{item coarse} The stratum $Z_{P_0,w}$ has coarse closure.
\end{enumerate}
\end{definition}
Using the notations of \Prop~\ref{comdiag}, property \eqref{item isom} is equivalent to $S\subset P_0\times G$. The justification of this definition is the following: For $P_0=P$, Condition \eqref{item isom} is obviously satisfied. On the other hand, if $P_0={}^z B$, then \eqref{item coarse} is satisfied because coarse and fine strata coincide. For a given $w$, a canonical parabolic for $w$ is an intermediate parabolic subgroup $P_0$ satisfying both conditions. A priori neither the existence nor the uniqueness of such a parabolic is clear.

We give justification for this definition. Let $P_0$ be a canonical parabolic subgroup for $w$. We have morphisms:
\begin{equation}
\pi:\overline{Z}_{P_0,w}\to \overline{Z}_w, \quad \tilde{\pi}:\overline{G}_{P_0,w}\to \overline{G}_w
\end{equation}
which yield isomorphisms $Z_{P_0,w}\simeq Z_w$ and $G_{P_0,w}\simeq G_w$. Since $Z_{P_0,w}$ has coarse closure, the stack (resp. variety) $\overline{Z}_{P_0,w}$ (resp. $\overline{G}_{P_0,w}$) is normal. We deduce the following proposition:

\begin{proposition}\label{normalization canonical}
Let $P_0$ be a canonical parabolic subgroup for $w\in {}^I W$. Write $w=xw_J$ as in \eqref{uniqueform}. Then the normalization of $\overline{G}_w$ is the Stein factorization of the map $\tilde{\pi}:\overline{G}_{P_0,w}\to \overline{G}_w$. It is isomorphic to $\Spec(\Ocal(\overline{\GG}_{P_0,x}))$, where
\begin{equation}
\overline{\GG}_{P_0,x}=\overline{G}_{P_0,w}=\{(g,aP_0)\in G\times P/P_{0}, \ a^{-1}g\varphi(\overline{a})\in \overline{P_0z\dot{w}Q_0} \}
\end{equation}
and the first projection induces an isomorphism $G_{P_0,w}\simeq G_w$.
\end{proposition}

The following theorem is the main result of this paper. Its proof will follow from the results of \S\ref{sec criterion i} and \S\ref{sec criterion ii}.
\begin{theorem}\label{mainthm}
Let $w\in {}^I W$. The parabolic subgroup $P_w$ is the unique canonical parabolic subgroup for $w$. More precisely, among all parabolic subgroups ${}^zB\subset P_0 \subset P$, the following holds:
\begin{enumerate}[(a)]
\item $P_w$ is the smallest parabolic $P_0$ such that $\pi:Z_{P_0,w}\to Z_w$ is an isomorphism.
\item $P_w$ is the largest parabolic $P_0$ such that $Z_{P_0,w}$ has coarse closure.
\end{enumerate}
\end{theorem}

\subsection{A criterion for Condition (i)}\label{sec criterion i}

\begin{lemma}\label{lemma criterion isom}
Let ${}^zB\subset P_0 \subset P$ be a parabolic subgroup. The following assertions are equivalent:
\begin{enumerate}[(1)]
\item \label{item lemma isom} The map $\pi:Z_{P_0,w}\to Z_w$ is an isomorphism.
\item \label{item inclusion} One has $P_w\subset P_0$.
\end{enumerate}
\end{lemma}

\begin{proof}
Using the notation of \Prop~\ref{comdiag}, we know that $\pi:Z_{P_0,w}\to Z_w$ is an isomorphism if and only if $S_{P_0}:=S \cap (P_0\times G)=S$. By the same proposition, we know that the quotient $S/S_{P_0}$ is a finite affine \'{e}tale scheme over $k$. In particular, we have $S\subset P_0\times G$ if and only if $S_{\rm red}\subset P_0\times G$.

By Lemma \ref{ptstab}, we can write $S_{\rm red}=A \ltimes R$ with $A$ the finite group given by equation \eqref{groupA} of Lemma \ref{ptstab} and $R$ a smooth unipotent connected normal subgroup. Write $R_{P_0}:=R\cap (P_0\times G)$. The inclusion $R\subset S$ induces a closed embedding
\begin{equation}
R/R_{P_0} \to S/S_{P_0}.
\end{equation}
Hence $R/R_{P_0}$ is a finite, smooth, connected $k$-scheme, so $R/R_{P_0}=\Spec(k)$, hence $R\subset P_0\times G$. It follows that $S\subset P_0\times G$ if and only if $A\subset P_0\times G$, which is equivalent to $A_1\subset P_0$, where 
\begin{equation}\label{inclusion A1}
A_1:=\{x\in L_w, \ {}^{z\dot{w}}\varphi(x)=x\}.
\end{equation}
By Steinberg's theorem we can write $z\dot{w}=a^{-1}\varphi(a)$ with $a\in G(k)$. Then it is easy to see that the subgroup ${}^{a}L_w$ is defined over $\FF_p$ and the inclusion $A_1\subset P_0$ is equivalent to
\begin{equation}\label{new inclusion}
({}^{a}L_w)(\FF_p) \subset {}^{a}P_0.
\end{equation}
Note that both ${}^{a}L_w$ and ${}^{a}P_0$ contain the torus ${}^{az}T$, which is defined over $\FF_p$. Thus Lemma \ref{Fp lemma} below shows that \eqref{new inclusion} is equivalent to ${}^{a}L_w \subset {}^{a}P_0$, hence $L_w\subset P_0$, which is the same as $P_w\subset P_0$. This terminates the proof.
\end{proof}

\begin{lemma}\label{Fp lemma}
Let $G$ be a connected reductive group over $\FF_p$. Let $L$ be a Levi $\FF_p$-subgroup of $G$ and $P$ be a parabolic subgroup of $G_{k}$. Assume that there exists a maximal $\FF_p$-torus $T$ contained in $L\cap P$ and that $L(\FF_p)\subset P$. Then one has $L\subset P$.
\end{lemma}

\begin{proof}
Define a subgroup of $G$ by
$$H:=L\cap \bigcap_{i\in \ZZ}\sigma^i(P)= \bigcap_{i\in \ZZ}L\cap \sigma^i(P).$$
It is clear that $H\subset L$, $H$ is defined over $\FF_p$, and $L(\FF_p)=H(\FF_p)$. Furthermore, $H$ is an intersection of parabolic subgroups of $L$ containing $T$. Hence it suffices to prove the following claim: Let $G$ be a connected reductive group over $\FF_p$, $T\subset G$ a maximal $\FF_p$-torus, and $T\subset H\subset G$ an $\FF_p$-subgroup which is an intersection of parabolic subgroups of $G_k$ containing $T$, and assume that $H(\FF_p)=G(\FF_p)$. Then one has $H=G$.

We now prove the claim. Using inductively \cite[\Prop~2.1]{digne-michel}, one shows that an intersection of parabolic subgroups $P_1,...,P_m$ containing $T$ is connected and can be written as a semidirect product 
\begin{equation}
\bigcap_{i=1}^m P_i=L_0 \ltimes U_0
\end{equation}
where $L_0$ is a Levi subgroup of $G$ containing $T$ and $U_0$ is a unipotent connected subgroup of $G$, normalized by $L_0$. Applying this to $H$, we can write $H=L_0\ltimes U_0$. Since $H$ is defined over $\FF_p$, so are $L_0$ and $U_0$.

By \cite[\Th~3.4.1]{Carter-Lie-type}, the highest power of $p$ dividing $|G(\FF_p)|$ is $p^N$, where $N=|\Phi_+|$ is the dimension of any maximal unipotent subgroup of $G_k$. Since $G(\FF_p)=H(\FF_p)=L_0(\FF_p)\times U_0(\FF_p)$, we deduce that for all maximal unipotent subgroup $U'$ in $L_0$, the subgroup $U'\times U_0$ is unipotent maximal in $G$. In particular, $H$ contains a Borel subgroup, so $H$ is a parabolic subgroup. Then $H=G$ follows from \cite[XXVI, 5.11]{SGA3}

\end{proof}

\subsection{A criterion for Condition (ii)}\label{sec criterion ii}
We examine Condition (ii) of Definition \ref{canonical}. Let ${}^zB\subset P_0 \subset P$ be a parabolic subgroup and let $L_0$, $M_0$, $Q_0$, $\Zcal_0$ as defined in section \ref{subsec fine coarse}. 
\begin{lemma} \label{lemma coarse}
Let ${}^zB\subset P_0 \subset P$ be a parabolic subgroup. The following assertions are equivalent:
\begin{enumerate}[(1)]
\item $Z_{P_0,w}$ has coarse closure.
\item One has ${}^{z\dot{w}}M_0=L_0$.
\end{enumerate} 
\end{lemma}

\begin{proof}
The stratum $Z_{P_0,w}$ has coarse closure if and only if $\Ocal_{E_{\Zcal_0}}(z\dot{w})\subset P_0z\dot{w}Q_0$ is an open embedding, which is equivalent to the equality of their dimensions. Note that $(B,T,z)$ is again a frame of $\Zcal_0$, so formula \eqref{dimzipstrata} shows that
\begin{equation}
\dim (\Ocal_{E_{\Zcal_0}}(z\dot{w}))=\dim(P_0)+\ell(w).
\end{equation}
On the other hand, we have:
\begin{equation}\label{dimbruhat}
\dim(P_0z\dot{w}Q_0)=2\dim(P_0)-\dim(\Stab_{P_0\times Q_0}(z\dot{w})).
\end{equation}
The stabilizer $\Stab_{P_0\times Q_0}(z\dot{w})$ is the subgroup:
\begin{align*}
\Stab_{P_0\times Q_0}(z\dot{w})&=\{(a,b)\in P_0\times Q_0, \ az\dot{w}=z\dot{w}b\} \\
&\simeq \{a\in P_0, \ (z\dot{w})^{-1}az\dot{w}\in Q_0 \}\\
&=P_0\cap{}^{z\dot{w}}Q_0.
\end{align*}
Hence $Z_{P_0,w}$ has coarse closure if and only if $\dim\left(P_0/(P_0\cap {}^{z\dot{w}}Q_0)\right)=\ell(w)$. Since the property is satisfied when $P_0={}^z B$, we have $\dim\left( {}^zB/({}^zB\cap {}^{z\dot{w}}B)\right)=\ell(w)$, so we can rewrite the property as
\begin{equation}\label{equivcond}
\dim\left((P_0\cap{}^{z\dot{w}}Q_0) / ({}^zB\cap {}^{z\dot{w}}B)\right)=\dim\left(P_0/ {}^{z}B\right)
\end{equation}
Since $(B,T,z)$ is a frame for $\Zcal_0$ and ${}^I W \subset {}^{I_0}W$, equation \eqref{IWeq} shows that $P_0\cap {}^zB=P_0\cap {}^{z\dot{w}}B$, thus the inclusion $P_0\cap{}^{z\dot{w}}Q_0\subset P_0$ induces an embedding
\begin{equation}
(P_0\cap{}^{z\dot{w}}Q_0) / ({}^zB\cap {}^{z\dot{w}}B)\rightarrow P_0/ {}^{z}B.
\end{equation}
Hence \eqref{equivcond} is satisfied if and only if the image of $P_0\cap {}^{z\dot{w}}Q_0$ is open in $P_0/{}^zB$. Since $P_0/{}^{z}B\simeq L_0/({}^{z}B\cap L_0)$ it is also equivalent to $L_0\cap {}^{z\dot{w}}Q_0$ having open image in $L_0/({}^zB\cap L_0)$.

Denote by $B'$ the opposite Borel in $G$ of $B$ with respect to $T$. Then ${}^z B'\cap L_0$ is the opposite Borel of ${}^zB\cap L_0$ in $L_0$ with respect to ${}^z T$. Thus the image of $L_0\cap {}^{z\dot{w}}Q_0$ is open in $L_0/({}^z B\cap L_0)$ if and only if ${}^zB'\cap L_0 \subset {}^{z\dot{w}}Q_0$. It follows immediately from equation \eqref{IWeq} that ${}^zB'\cap L_0 = {}^{z\dot{w}}B'\cap L_0$. Finally, we find that $Z_{P_0,w}$ has coarse closure if and only if
\begin{equation}\label{equivcond2}
B'\cap {}^{(z\dot{w})^{-1}}L_0 \subset Q_0.
\end{equation}
The groups $B'\cap {}^{(z\dot{w})^{-1}}L_0$ and $B\cap {}^{(z\dot{w})^{-1}}L_0$ are opposite Borel subgroups of ${}^{(z\dot{w})^{-1}}L_0$ containing $T$. Since $B\subset Q_0$, equation \eqref{equivcond2} is simply equivalent to ${}^{(z\dot{w})^{-1}}L_0 \subset Q_0$, which is equivalent to ${}^{(z\dot{w})^{-1}}L_0 = M_0$. This terminates the proof.
\end{proof}

\begin{proof}[Proof of Theorem \ref{mainthm}]
The result follows immediately by combining Lemmas \ref{lemma criterion isom} and \ref{lemma coarse}.
\end{proof}

\begin{corollary}\label{cor norma GZip}
Write $w=xw_J$ as in \eqref{uniqueform}. The normalization of the Zariski closure $\overline{G}_w$ is the Stein factorization of the map $\tilde{\pi}:\overline{G}_{P_w,w}\to \overline{G}_w$. It is isomorphic to $\tilde{Z}_w:=\Spec(\Ocal(\overline{\GG}_{P_w,x}))$, where
\begin{equation}
\overline{\GG}_{P_w,x}=\overline{G}_{P_w,w}=\{(g,aP_w)\in G\times P/P_{w}, \ a^{-1}g\varphi(\overline{a})\in \overline{P_wz\dot{w}Q_w} \}
\end{equation}
and the first projection induces an isomorphism $G_{P_w,w}\simeq G_w$.
\end{corollary}

\subsection{Shimura varieties and Ekedahl-Oort strata} \label{subsec Shimura}
Let $X$ be the special fiber of a Hodge-type Shimura variety attached to a Shimura datum $(\GG,\XX)$ with hyperspecial level at $p$. Write $G:=G_{\ZZ_p}\times \FF_p$, where $G_{\ZZ_p}$ is a reductive $\ZZ_p$-model of $\GG_{\QQ_p}$. By Zhang \cite{ZhangEOHodge}, there exists a smooth morphism of stacks
\begin{equation}
\zeta:X\longrightarrow \GZip^\Zcal
\end{equation}
where $\Zcal$ is the zip datum attached to $(\GG,\XX)$ as in \cite{Goldring-Koskivirta-zip-flags}~\S6.2. The Ekedahl-Oort stratification of $X$ is defined as the fibers of $\zeta$. For $w\in {}^IW$, set:
\begin{equation}
X_{w}:=\zeta^{-1}(Z_{w}).
\end{equation}
By the smoothness of $\zeta$, this defines a stratification of $X$. Let ${}^z B \subset P_0 \subset P$ be a parabolic subgroup and define the partial flag space $X_{P_0}$ as the fiber product
\begin{equation} \label{eq-def-flag-space}
\xymatrix@1@M=5pt{
X_{P_0} \ar[r]^-{\zeta_{P_0}} \ar[d]_-{\pi} & \GF^{(\Zcal, P_0)} \ar[d]^-{\pi_{P_0}} \\
X \ar[r]_-{\zeta} & \GZip^{\Zcal}}
\end{equation} 
The map $\pi:X_{P_0}\to X$ is a $P/P_0$-bundle. For $w\in {}^{I_0}W$ and $x\in {}^{I_0}W^{J_0}$ define
\begin{align}
X_{P_0,w}&:=\zeta_0^{-1}(Z_{P_0,w}) \\
\XX_{P_0,x}&:=\zeta_0^{-1}(\ZZ_{P_0,x}).
\end{align}
We call $X_{P_0,w}$ the fine stratum attached to $w\in {}^{I_0}W$ and $\XX_{P_0,x}$ the coarse stratum attached to $x\in {}^{I_0}W^{J_0}$. All coarse and fine strata are smooth and locally closed, they define stratifications of $X_{P_0}$, and the Zariski closure of a coarse stratum is normal. Recall that we defined $\tilde{Z}_w:=\Spec(\Ocal(\overline{G}_{P_w,w}))$ (\Cor~\ref{cor norma GZip}).

\begin{corollary}\label{cor norma Shimura}
Let $w\in {}^IW$. The normalization of $\overline{X}_{w}$ is the Stein factorization of the map $\pi: \overline{X}_{P_w,w}\to\overline{X}_{w}$. It is isomorphic to $\overline{X}_{w}\times_{\GZip^\Zcal}\tilde{Z}_w$.
\end{corollary}

\section{The canonical filtration}

Most of the content of this section can be found in \cite{Pink-Wedhorn-Ziegler-zip-data}. We merely unwind their proofs to make the link between the canonical filtration of a Dieudonn\'{e} space and the group $L_w$ defined previously. See also \cite[\S 4.4]{moonen-gp-schemes} and \cite{Boxer-thesis} for related results.

\subsection{Dieudonn\'{e} spaces and $GL_n$-zips}\label{subsec Dieudonne zips}

Let $H$ be a truncated Barsotti-Tate groups of level $1$ over $k$ of height $n$. Set $d:=\dim(\Lie(H))$ and write $D:=\DD(H)$ for its Dieudonn\'{e} space. It is a $k$-vector space of dimension $n$ endowed with a $\sigma$-linear endomorphism $\Fcal:D\to D$, a $\sigma^{-1}$-linear endomorphism $\Vcal:D\to D$ satisfying the conditions:
\begin{enumerate}[(1)]
\item $\Ker(\Fcal)=\Im(\Vcal)$,
\item $\Ker(\Vcal)=\Im(\Fcal)$,
\item $\rk(\Vcal)=d.$
\end{enumerate}
We say that $(D,\Fcal,\Vcal)$ is a Dieudonn\'{e} space of height $n$ and dimension $d$. Let $M^{(r)}_n(k)$ be the set of matrices in $M_n(k)$ of rank $r$. After choosing a $k$-basis of $D$, we may write $\Fcal=a\otimes \sigma$ and $\Vcal=b\otimes \sigma^{-1}$, where $(a,b)$ is in the set
$$\Xscr:=\{(a,b)\in M^{(n-d)}_n(k)\times M^{(d)}_n(k), \ a\sigma(b)=\sigma(b)a=0\}.$$

Note that for $(a,b)\in \Xscr$, we have $\Ker(a)=\Im(\sigma(b))=\sigma(\Im(b))$ and $\Im(a)=\Ker(\sigma(b))=\sigma(\Ker(b))$.

It is easy to see that two such pairs $(a,b)$ and $(a',b')$ yield isomorphic Dieudonn\'{e} spaces if and only if there exists $M\in GL_n(k)$ such that
\begin{equation}
a'=Ma\sigma(M)^{-1} \quad \textrm{ and } \quad b'=Mb\sigma^{-1}(M)^{-1}.
\end{equation}
This defines an action of $GL_n(k)$ on $\Xscr$ and we obtain a bijection between isomorphism classes of Dieudonn\'{e} spaces of height $n$ and dimension $d$ and the set of $GL_n(k)$-orbits in $\Xscr$. 

Let $(e_1,...,e_n)$ the canonical basis of $k^n$ and define
\begin{align}
V_1&:=\Span(e_1,...,e_{n-d})\\
V_2&:=\Span(e_{n-d+1},...,e_n)
\end{align}
Define $P:=\Stab(V_2)$, $Q:=\Stab(V_1)$, $L:=P\cap Q$, $U:=R_u(P)$ and $V:=R_u(Q)$. Consider the set
\begin{equation}
\Yscr:=\{ (a,b)\in \Xscr, \ \Ker(a)=V_2\}.
\end{equation}
The action of $GL_n(k)$ on $\Xscr$ restricts to an action of $P(k)$ on $\Yscr$ and the inclusion $\Yscr \subset \Xscr$ induces a bijection between $P(k)$-orbits in $\Yscr$ and $GL_n(k)$-orbits in $\Xscr$.

\begin{lemma} There is a natural bijection
\begin{equation}
\Psi:\Yscr \longrightarrow GL_n(k)/V
\end{equation}
\end{lemma}

\begin{proof}
Let $(a,b)\in \Yscr$ and choose a subspace $H\subset k^n$ such that $\Im(a)\oplus H=k^n$. Define a matrix $f_H\in GL_n(k)$ by the following diagram
\begin{equation}\label{diagram fH}
\xymatrix@C=1pt{
k^n \ar[d]_{f_H}&=&V_1 \ar[d]^{a} & \oplus & V_2  \ar[d]^{\sigma(b)^{-1}}\\
k^n &=& \Im(a)& \oplus &H}.
\end{equation}
In other words, $f_Hv=av$ for $v\in V_1$, and if $v\in V_2$, then $f_Hv$ is the only element $h\in H$ such that $\sigma(b)h=v$ (note that $\Im(a)=\Ker(\sigma(b))$, so this element is well-defined). It is clear that $f_H$ is invertible.

If $H'$ denotes another subspace such that $\Im(a)\oplus H'=k^n$, then we can write $f_{H'}= f_H\alpha$, for some $\alpha \in GL_n(k)$. It is clear that $\alpha(v)=v$ for all $v\in V_1$. Furthermore, for $v\in V_2$, one must have $\sigma(b)(f_{H'}v-f_Hv)=0$, thus $f_H(\alpha(v)-v)\in \Ker(\sigma(b))=\Im(a)$, so $\alpha(v)-v\in V_1$. This shows that $\alpha \in V$. It follows that $(a,b)\mapsto f_H$ induces a well-defined map
$\Psi:\Yscr \longrightarrow GL_n(k)/V$. We leave it to the reader to check that this map is a bijection.
\end{proof}

Define a subgroup of $P\times Q$ by
\begin{equation}
E:=\{(M_1,M_2)\in P\times Q, \ \varphi(\overline{M}_1)=\overline{M}_2\}.
\end{equation}
Let this group acts on $GL_n(k)$ by the rule $(M_1,M_2)\cdot g:=M_1gM_2^{-1}$.

\begin{proposition} The map $\Psi$ induces a bijection
\begin{equation}
P(k)\backslash \Yscr \longrightarrow E\backslash GL_n(k)
\end{equation}
Hence there is a bijection between isomorphism classes of Dieudonn\'{e} spaces of height $n$ and dimension $d$ and the set of $E$-orbits in $GL_n(k)$.
\end{proposition}

\begin{proof}
Let $M\in P(k)$, $(a,b)\in \Yscr$ and set $(a',b'):=(Ma\sigma(M)^{-1},Mb\sigma^{-1}(M)^{-1})$. Note that $\Im(a')=M(\Im(a))$. Choose a subspace $H$ such that $\Im(a)\oplus H=k^n$ and set $H':=M(H)$. Let $\overline{M}\in L(k)$ denote the Levi component of $M\in P(k)$. Finally, write $f_H$ and $f'_{H'}$ for the maps attached to $(a,b,H)$ and $(a',b',H')$ respectively by the previous construction. We claim that one has the relation:
\begin{equation}\label{transform}
Mf_H=f'_{H'}\sigma(\overline{M}).
\end{equation}
First assume $v\in V_1$. Then $f'_{H'}\sigma(\overline{M})v=Ma\sigma(M)^{-1}\overline{M}v$. Since $\sigma(M)^{-1}\overline{M}\in U$, we have $\sigma(M)^{-1}\overline{M}v-v \in V_2$, hence $f'_{H'}\sigma(\overline{M})v=Mav=Mf_Hv$.

Now if $v\in V_2$, the element $f_{H}v$ is the only element $h\in H$ satisfying $\sigma(b)h=v$. Similarly, $f_{H'}\sigma(\overline{M})v$ is the only element $h'\in H'=M(H)$ such that $\sigma(b')h'=\sigma(\overline{M})v$. Hence $\sigma(\overline{M})^{-1}\sigma(M)\sigma(b)M^{-1}h'=v$. But $\sigma(\overline{M})^{-1}\sigma(M) \in U$ and $\sigma(b)M^{-1}h'\in V_2$, so we deduce $\sigma(b)M^{-1}h'=v$, and finally $M^{-1}h'=h$ as claimed.

This shows that $\Psi$ induces a well-defined map $P(k)\backslash \Yscr \to E\backslash GL_n(k)$. We leave it to the reader to check that it is bijective.
\end{proof}

\subsection{The canonical filtration}
Let $(D,\Fcal,\Vcal)$ be a Dieudonn\'{e} space. The operators $\Vcal$ and $\Fcal^{-1}$ act naturally on the set of subspaces of $D$. It can be shown that there exists a flag of $D$ which is stable by $\Vcal$ and $\Fcal^{-1}$ and which is coarsest among all such flags. This flag is called the canonical filtration of $D$. It is obtained by applying all finite combinations of $\Vcal,\Fcal^{-1}$ to the flag $0\subset D$.

Choose a basis of $D$ and write $\Fcal=a\otimes \sigma$ and $\Vcal=b\otimes \sigma^{-1}$ with $(a,b)\in \Xscr$. By choosing an appropriate basis, we will asssume that $(a,b)\in \Yscr$.
\begin{rmk}\label{zeroone}
Actually, there exists a basis such that $(a,b)\in \Yscr$ and such that the coefficents of $a,b$ are either $0$ or $1$ and each column and each row has at most one non-zero coefficient.
\end{rmk}
Let $H\subset k^n$ be a subspace such that $\Im(a)\oplus H=k^n$ and let $f_H\in GL_n(k)$ be the element defined in diagram \eqref{diagram fH}. We have the following easy lemma:

\begin{lemma}
For any subspace $W\subset k^n$, one has the following relations:
\begin{align}
\Vcal(W)&=V_2 \cap \left(\sigma^{-1}(f_H^{-1}W)+V_1\right) \\
\Fcal^{-1}(W)&=V_2+\left(\sigma^{-1}(f_H^{-1}W)\cap V_1\right).
\end{align}
\end{lemma}

In particular, the right-hand terms are independent of the choice of $H$. This observation suggests the following definition: 
\begin{definition}
For $f\in GL_n(k)$, there exists a unique coarsest flag $\Fl(f)$ of $k^n$ satisfying the following properties:
\begin{enumerate}[(1)]
\item For any $W\in \Fl(f)$, the following inclusions hold 
\begin{equation}\label{inclusions}
V_2 \cap \left(\sigma^{-1}(f^{-1}W)+V_1\right) \subset W \subset V_2+\left(\sigma^{-1}(f^{-1}W)\cap V_1\right).
\end{equation}
\item For any $W\in \Fl(f)$, all subspaces appearing in \eqref{inclusions} are in $\Fl(f)$.
\end{enumerate}
\end{definition}
The flag $\Fl(f)$ is simply the canonical flag attached to the Dieudonn\'{e} space corresponding to the left-coset $fV$ under the bijection $\Psi$.

\begin{lemma}
Let $f\in GL_n(k)$. The following assertions hold
\begin{enumerate}[(1)]
\item For all $v\in V$, one has $\Fl(fv)=\Fl(f)$.
\item For $M\in P$, one has $\Fl(Mf\sigma(\overline{M})^{-1})=M\Fl(f)$.
\end{enumerate}
\end{lemma}

We leave the verification of the lemma to the reader. In particular, the conjugation class of $\Fl(f)$ depends only on the $E$-orbit of $f$. Denote by $P(f):=\Stab(\Fl(f))$. Since $\Fl(f)$ contains $\Im(V)=V_2$, we have $P(f) \subset P$. Furthermore, for $v\in V$ and $M\in P$, one has
\begin{align}
P(fv)&=P(f)  \label{parab change 1} \\
P(Mf\sigma(\overline{M})^{-1})&={}^M P(f). \label{parab change 2}
\end{align}

\subsection{The canonical flag versus $P_w$}
Denote by $T$ the diagonal torus of $G:=GL_n$, and let $B$ be the Borel subgroup of upper-triangular matrices. The Weyl group $W(G,T)$ is the symmetric group $S_n$, which we identify with a subgroup of $G(k)$ be letting it act on $k^n$ by $\tau(e_i)= e_{\tau(i)}$ for all $\tau \in S_n$ and $i\in \{1,...,n\}$.

Using the notations of section \ref{subsec frame}, define a permutation
\begin{equation}
z:=w_0 w_{0,I}=\left(\begin{matrix}
0&I_{n-d} \\
I_d&0
\end{matrix}\right).
\end{equation}
Then $(B,T,z)$ is a frame for the zip datum $(G,P,Q,L,M,\varphi)$. For $w\in {}^I W$, set $f_w:=zw$. By the parametrization \eqref{param}, the set $\{zw,w\in {}^I W\}$ is a set of representatives of the $E$-orbits in $G$. For $w\in {}^IW$, it is easy to see that any $W\in \Fl(f_w)$ is spanned by $(e_i)_{i\in C_W}$ for some subset $C_W\subset \{1,...,n\}$. In particular we have $T\subset P(f_w)$. Note that for all $w\in {}^IW$, we have simplified formulas:
\begin{align}
V_2 \cap \left(\sigma^{-1}(f_w^{-1}W)+V_1\right)&=V_2 \cap f_w^{-1}W \\
V_2+\left(\sigma^{-1}(f_w^{-1}W)\cap V_1\right)&=V_2+f_w^{-1}W.
\end{align}

There is a unique Levi subgroup $L(f_w)\subset P(f_w)$ containing $T$. Finally, for $w\in {}^IW$, denote by $L_w\subset L$ and $P_w\subset P$ the subgroups defined in \eqref{canlevi} and \eqref{canparab}.

\begin{proposition}\label{can prop}
We have $P(f_w)=P_w$ and $L(f_w)=L_w$.
\end{proposition}
\begin{proof}
We will show first that ${}^zB\subset P(f_w)$. Clearly it suffices to show ${}^zB\cap L\subset P(f_w)$. Note that since $w\in {}^IW$, we have ${}^zB\cap L=B\cap L = {}^{zw}B \cap L$. From this it follows that if $W\subset k^n$ is a subspace such that ${}^zB\cap L\subset \Stab(W)$, then ${}^zB\cap L$ stabilizes also $\sigma^{-1}(f_w^{-1}(W))$. From this it follows easily by induction that ${}^zB\cap L$ stabilizes $\Fl(f_w)$, hence ${}^zB\cap L\subset P(f_w)$ as claimed. 

To finish the proof, it suffices to show the second assertion. By definition we have ${}^{f_w}\varphi(L_w)=L_w$. Hence if $W\subset k^n$ is a subspace such that $L_w\subset \Stab(W)$ then $L_w\subset \Stab(\sigma^{-1}(f_w^{-1}(W)))$. From this, it follows again by an easy induction that $L_w$ stabilizes $\Fl(f_w)$, so $L_w\subset P(f_w)$. Since $L_w$ contains the torus $T$, we deduce that $L_w\subset L(f_w)$.

Finally, we must show that ${}^{f_w}\varphi(L(f_w))=L(f_w)$. Since $L(f_w)$ is clearly defined over $\FF_p$, this is the same as ${}^{f_w}L(f_w)=L(f_w)$. Let $k^n=D_1\oplus...\oplus D_m$ denote the decomposition attached to $L(f_w)$, numbered so that the filtration $\Fl(f_w)$ is composed of the subspaces $W_j:=\bigoplus_{j=1}^i D_j$ for $1\leq i \leq m$. There exists an integer $1\leq r \leq m$ such that $W_r=V_2$ (and then necessarily $V_1=\bigoplus_{j=r+1}^m D_j$). We need to show that $f_w$ permutes the $(D_i)_{1\leq i \leq m}$. For this, it suffices to show that if $f_w(D_i)\cap D_j\neq 0$, then $D_j\subset f_w(D_i)$ for all $1\leq i,j\leq m$.

First assume $1\leq j \leq r$ and let $1\leq i \leq m$ be the smallest integer such that $D_j\cap f_w(D_i)\neq 0$. He have $0\neq D_j\cap f_w(D_i)\subset V_2\cap f_w(W_i)$, which implies $D_j\subset V_2\cap f_w(W_i)$. By minimality of $i$, we deduce $D_j\subset f_w(D_i)$.

Now assume $r< j \leq m$ and let $1\leq i \leq m$ be the smallest integer such that $D_j\cap f_w(D_i)\neq 0$. Then $0\neq D_j\cap f_w(D_i)\subset V_2+ f_w(W_i)$, which implies $D_j\subset (V_2+ f_w(W_i))\cap V_1= f_w(W_i)\cap V_1\subset f_w(W_i)$. By minimality of $i$, we deduce $D_j\subset f_w(D_i)$, which terminates the proof.

\end{proof}

\bibliographystyle{amsalpha}
\bibliography{biblio_overleaf}

\end{document}